\documentclass[10pt]{amsart}

\usepackage{graphicx}
\usepackage{xcolor}

\newcommand{\ra}{\rightarrow}
\newcommand{\ve}{\varepsilon}

\newcommand{\8}{\infty}
\newcommand{\nn}{\nonumber}
\newcommand{\be}{\begin{eqnarray}}
\newcommand{\ee}{\end{eqnarray}}

\newtheorem{thm}{Theorem}
\newtheorem{prop}[thm]{Proposition}
\newtheorem{lem}[thm]{Lemma}
\newtheorem{dfn}[thm]{Definition}

\newtheorem{rmk}[thm]{Remark}

\newtheorem{example}{Example}

\newcommand{\R}{\mathbb{R}}
\newcommand{\T}{\mathbb{T}}
\newcommand{\Z}{\mathbb{Z}}
\newcommand{\N}{\mathbb{N}}
\newcommand{\Q}{\mathbb{Q}}
\newcommand{\D}{\mathbb{D}}

\newcommand{\A}{\mathcal{A}}
\newcommand{\AN}{\mathbb{A}}
\newcommand{\M}{\mathcal{M}}
\newcommand{\tg}{\tilde{\gamma}}

\title{Action and periodic orbits on annulus }
\author{Yanxia Deng, Zhihong Xia}
\address{School of Mathematics (Zhuhai), Sun Yat-sen University, Zhuhai, Guangdong, China}
\address{Department of Mathematics, Northwestern University, Evanston, IL 60208 USA}

\email{dengyx53@mail.sysu.edu.cn, xia@math.northwestern.edu}

\date{version: June 10, 2021}

\begin{document}

\maketitle 

\begin{abstract}
  We consider the classical problem of area-preserving maps on annulus
  $\AN = S^1 \times [0, 1]$ . Let $\M_f$ be the set of all invariant
  probability measures of an area-preserving, orientation preserving
  diffeomorphism $f$ on $\AN$. Given any $\mu_1$ and $\mu_2$ in
  $\M_f$, Franks \cite{Franks1988}\cite{Franks1992}, generalizing Poincar\'e's last
  geometric theorem (Birkhoff \cite{Birkhoff1913}), showed that if their
  rotation numbers are different, then $f$ has infinitely many
  periodic orbits. In this paper, we show that if $\mu_1$ and $\mu_2$
  have different actions, even if they have the same rotation number,
  then $f$ has infinitely many periodic orbits.  In particular,
  if the action difference is larger than one, then $f$ has at least
  two fixed points. The same result is also true for area-preserving
  diffeomorphisms on unit disk, where no rotation number is available.

  %Moreover, we show that, given any invariant measure $\mu$, the
  %action of $\mu$ is a limit of actions on periodic orbits. This
  %result extends a recent theorem of Hutchings \cite{Hutchings2016} on
  %Calabi invariants on disks and of Weiler \cite{Weiler2019} on
  %annulus.
  
\end{abstract}

\section{Introduction}

Let $\AN= S^1 \times [0, 1]$, where $S^1 = \mathbb{R} / \mathbb{Z} $,
be an annulus with the standard area form $\omega = dy \wedge dx$,
where $x \in S^1, y \in [0, 1]$, and let $f$ be an area-preserving,
orientation preserving diffeomorphism on $\AN$ that preserves each
boundary component. Let $\beta$ be a primitive of $\omega$,
i.e. $d\beta=\omega$. Since $f$ is area-preserving and preserves
each boundary component, we know $f^*\beta-\beta$ is an exact 1-form. There is
a function $g$ on $\AN$ such that \[dg=f^*\beta-\beta.\] The real
valued function $g$ is called the {\it action function} for $f$. The
choice of $g$ depends on two factors, the first one is the integration
constant, which can be fixed by assigning a zero value at a
particular point. The second factor is the choice of $\beta$. Two
different choices of $\beta$ differ by a closed 1-form, we will show
how the action depends on $\beta$ (Proposition \ref{prop4}).  Let
$\M_f$ be the set of all $f$-invariant probability measures on
$\AN$. For any $\mu \in \M_f$, the {\em mean action}\/, or simply the
action, of $\mu$ is defined to be
$$\mathcal{A}(\mu) = \int_\AN gd\mu.$$
When $\mu$ is the area form, the action is called the {\em
  Calabi invariant}\/ (cf. \cite{Hutchings2016}).

\

Another important quantity that is associated with an invariant
measure $\mu \in \M_f$ is its {\em rotation number}\/. It measures the
average rotation around the annulus. The precise definition of
rotation numbers, and rotation vectors for general manifold, will be
given in the next section. Given any two invariant measures $\mu_1$
and $\mu_2$ in $\M_f$, Franks \cite{Franks1988, Franks1992},
generalizing Poincar\'e's last geometric theorem (Poincar\'e-Birkhoff
Theorem \cite{Birkhoff1913}), showed that if the rotation numbers of
$\mu_1$ and $\mu_2$ are different, then $f$ has infinitely many
periodic orbits. More precisely, for any rational number $p/q$ between
the rotation numbers of any two invariant measures, where $p$ and $q$
are relatively prime, there are at least two distinct periodic orbits
of period $q$ with rotation number $p/q$.

Poincar\'e's last geometric theorem, more generally Franks' theorem,
shows that the difference in rotation numbers forces the existence of
many other periodic orbits. In this paper, we will show that the same
is true for differences in actions. But first, we need to point out
that the difference in actions $\A(\mu_1) - \A(\mu_2)$ for two
invariant measures, in general, depends on the choice of $\beta$,
defined by $d\beta=\omega$. Interestingly, this dependence is closely
related to their rotation numbers. This dependence will be made
precise in the next section (Proposition \ref{prop4}). In particular,
we will show that if two invariant measures $\mu_1$ and $\mu_2$ have
exactly the same rotation number, then the action difference
$\A(\mu_1) - \A(\mu_2)$ is independent of any choice of the 1-form
$\beta$ such that $d\beta = \omega$.

For the problem on the annulus, to precisely state our results, we
will fix the special 1-form $\beta = y dx$, under the standard
coordinate. The choice of this 1-form is equivalent to collapsing the
lower boundary component into a point, effectively killing the
underlying topology.

Our main result is

\begin{thm}
\label{thm:main}
Let $f$ be an area-preserving, orientation preserving diffeomoprhism
on $\AN$, isotopic to identity. Let $\mu_1, \mu_2 \in \M_f$ be any two
$f$-invariant probability measures. Suppose that
$|\A(\mu_1) -\A(\mu_2)| \neq 0$, then $f$ has infinitely many distinct
periodic points. More precisely, for any positive integer $q$ such that
$$q
> \frac{1}{|\A(\mu_1) -\A(\mu_2)|} ,$$
$f$ has at least two distinct
periodic orbits with period $q$, and $q$ is the least period if it is a prime number.  In
particular, if
$|\A(\mu_1) -\A(\mu_2)|>1$, then $f$ has at least two fixed points.

The same result is also true for area-preserving, orientation
preserving diffeomorphism of the unit disk $\D = \{ (x, y) \in \R^2,\; x^2 +
y^2 \leq 1\}$.
\end{thm}

As an application to our theorem, we consider a rigid rotation on the
annulus with an irrational rotation number. The action function in
this case is a constant. One can perturb the map, preserving the area,
in a neighborhood of an essential simple closed curve to change the
mean rotation number, therefore creating new periodic orbits by
Franks' theorem. Our result shows that this can be done locally. Pick
any point in the interior of the annulus. For any small neighborhood
around the point, we can easily change the map to increase or decrease
the mean action of the map, consequently we have a different action
with respect to the area, i.e., the Calabi invariant is now different
from the actions of untouched orbits. Therefore, there must be infinitely
many new periodic points. The details of this example will be given at
the last section of this paper.

\section{Action and rotation vectors}
In this section, we give a general introduction of the action and
rotation vectors of symplectic diffeomorphisms.

\subsection{Action function}
Let $M^{2n}$ be a $2n$-dimensional symplectic manifold with a
non-degenerate closed 2-form $\omega$, and let $$f: M \rightarrow M$$
be a diffeomorphism preserving the symplectic form $\omega$.  We further
assume that $f$ is exact symplectic, i.e., $f$ is isotopic to
identity; there is a 1-form $\beta$ such that $\omega = d \beta$, and
there is a function $g$ on $M$ such that $$dg = f^*\beta - \beta.$$
The real valued function $g$ is called the {\it action function} for
$f$. The choice of $g$ depends on two factors, the first one is the
integration constant.  This can be fixed by assigning $g$ a particular
value at a special point, say $x_0\in M$. The second factor is the
choice of $\beta$. Suppose $$\omega = d \beta = d \tilde{\beta},$$
then $\beta-\tilde{\beta}$ is a closed form. Let
$$[\beta-\tilde{\beta}] \in H^1(M, \R).$$
be the cohomology class of the difference. It turns out that this
cohomology class plays an essential role. 

First, let's suppose that the cohomology
class of $\beta-\tilde{\beta}$ is trivial, then  there is a
function $S: M \ra \R$, such that $$\tilde{\beta} - \beta = dS,$$ then
for any function $\tilde{g}$ such that
$$d\tilde{g} = f^*\tilde{\beta} - \tilde{\beta},$$
we have $$d\tilde{g} - d{g} =  (f^*\tilde{\beta} - f^*\beta)
+(\tilde{\beta} - \beta)  =  f^*d S - d S = d (S\circ f -  S)$$
or $$\tilde{g} - {g} = (S\circ f -  S) +C$$
for some constant $C$. Conversely, for any function $S$ on $M$, let
$$\tilde{g} = {g} + (S\circ f -  S) +C,$$
then $\tilde{g}$ is also an action function for $f$, with trivial cohomology
class for $\beta-\tilde{\beta}$ for corresponding 1-forms.

Two real valued functions $\tilde{g}$ and $g$ are said to be
{\em cohomologous}\/ if there is a real valued function $S$ on $M$
such that
$$\tilde{g} = {g} + (S\circ f -  S).$$

It is important to note that for any $f$-invariant measure $\mu$ on
$M$, $$\int_M (S\circ f - S) d\mu =\int_M Sd(f^*\mu) - Sd\mu =0, $$
in particular, $$\int_M (S\circ f - S) \omega^n =0. $$

It follows that, if $g$ and $\tilde{g}$ are cohomologous, then the
  action defined by these action functions are exactly same.

Next, we define the {\it mean action}\/ of $f$, the Calabi
invariant. If $M$ has bounded volume, we assume, without loss of
generality, rescale $\omega$ so that $\int_M  \omega^n=1$.  Then the
mean action is simply
$$\mathcal{A} (f) = \int_M g \omega^n.$$

We remark that the action, and therefore the mean action, defined
above is relative. It denpends on the integration constant $C$. If
neccessary, we will make proper choices to fix $C$.

The action also depends on the first cohomology of $M$. This
dependence is more interesting and will be explored later in the
section. For now, we will fix a 1-form $\beta$.

It is a very useful and convenient fact that the mean action is
additive for composition of diffeomorphisms.

\begin{prop} \label{prop:ac_comp}  Let $f_1$ and $ f_2$ be exact
  symplectic diffeomorphisms of $M$. Let 
$g_1$, $g_2$ and $g_{12}$ be action functions for $f_1$, $f_2$ and $f_2
\circ f_1$ respectively. Suppose there is a point $x_0 \in M$ such that
 $$g_{12}(x_0) = g_1(x_0) + g_2(f_1(x_0)).$$

Then $$\mathcal{A}(f_2 \circ f_1) = \mathcal{A}(f_1) +
  \mathcal{A}(f_2).$$ 
\end{prop}

\begin{proof} The action function $g_{12}$ is define by
$$dg_{12} = (f_2  \circ f_1)^* \beta -\beta,$$
hence,
$$dg_{12} =\left\{ f_1^*(f_2 ^* \beta) - f_2 ^* \beta \right\} + \{f_2 ^* \beta- \beta\}$$
$$=d\tilde{g}_1 + d g_2$$
for some function $\tilde{g}_1$. In fact  $$d\tilde{g}_1 = f_1^*(f_2 ^* \beta) - f_2 ^* \beta =
\{f_1^*(f_2^* \beta- \beta) - (f_2 ^* \beta - \beta)\} + \{f_1^*\beta -
\beta \}$$
$$=\{ f_1^*dg_2- dg_2\}  + dg_1= d(g_2\circ f_1-g_2) + dg_1.$$
Obviously, such $\tilde{g}_1$ exists, we may choose
$$\tilde{g}_1 = (g_2\circ f_1 -g_2) + g_1,$$
and then we may choose  $$g_{12} = \tilde{g}_1 + g_2 = (g_2\circ f_1 -g_2) + g_1 +
g_2 =  g_2\circ f_1  + g_1,$$
i.e., we may choose the action function for $f_2\circ f_1$ by adding
the action function for $f_1$ and the $f_1$-shifted action function for $f_2$. 
Under this choice, we have \be {g}_{12}(x_0) =  g_1(x_0) +
g_2(f_1(x_0)) \label{eq1}.\ee

Now \be \A(f_2 \circ f_1) &=& \int_M g_{12} \omega^n \nn \\
&=& \int_M (g_1 + g_2\circ f_1)
\omega^n \nn \\
&=& \int_M g_1 \omega^n + \int_M (g_2\circ f_1) 
\omega^n \nn \\
&=& \int_M g_1 \omega^n + \int_M g_2 (f_1^* 
\omega^n) \nn \\
&=& \int_M g_1 \omega^n + \int_M g_2
\omega^n \nn \\
&=& \A(f_1) + \A(f_2). \nn \ee
We remark that the above equality is independent of the choice of
1-form $\beta$ and therefore the choices of
the action functions, as long as condition (\ref{eq1}) holds.

This proves the proposition.
\end{proof}

\subsection{Action on invariant measures}
\label{sec:ac_invm}
 Let  $$f: M \rightarrow M$$ be an exact symplectic diffeomorphism and
 let $$g: M \ra \R$$ be a fixed action function for $f$. Let $\M_f$ be
 the set of all $f$-invariant probability measures on $M$. For any
 $\mu \in \M_f$, the {\em mean action}\/, or simply the action,  of $\mu$ is defined to be
 $$\mathcal{A}(\mu) = \int_M gd\mu.$$
 In particular, for any periodic orbit $\gamma = \{p_0, p_1, \ldots,
p_k=p_0\}$, where $f^i(p_0) = p_i$, for $i = 1, 2,\ldots, k$, the
corresponding invariant measure is
$$\mu_\gamma = \frac{1}{k} 
(\delta_{p_0} + \delta_{p_1} + \ldots + \delta_{p_{k-1}}), $$
and the mean action on $\gamma$ is
$$\mathcal{A}(\gamma) =\mathcal{A}(\mu_\gamma) = \frac{1}{k} (g(p_0) + g(p_1) + \ldots +
g(p_{k-1})).$$

\medskip

The action we defined so far is for exact symplectic
diffeomorphisms. However, it is a well-known fact that if $M$ is
compact, then there is no 1-form $\beta$ such that $d \beta = \omega$,
hence no symplectic diffeomorphism on a compact manifold is exact. A
class of symplectic diffeomorphisms that shares many properties with
exact symplectic diffeomorphisms is the {\em Hamiltonian
  diffeomorphisms}\/. These are the time-1 maps of periodic
Hamiltonian flows. On compact surfaces, one can blow up a contractible
fixed point, whose existence is guarantied by Arnold's conjecture, and
then consider the exact area-preserving map on a compact surface with
boundary.

On a compact symplectic manifold, we can also take another
approach. Since action on periodic orbits, and in extension, invariant
measures, is relative, we are more interested in the differences
between two invariant measures.

Let $p_1$ and $p_2$ be two fixed points for a symplectic
diffeomorphism $f: M \ra M$. We assume that $f$ is isotopic to
identity, but not necessarily exact.  We say that $p_1$ and $p_2$ are
homologous, if the loops $\tilde{\gamma}_1$ and $\tilde{\gamma}_2$,
starting with $p_1$ and $p_2$ respectively, following the isotopy of
$f$ back to $p_1$ and $p_2$ respectively, are homologous,
i.e., $$[\tg_1] = [\tg_2] \in H_1(M, \R)$$ Let $\gamma$ be a curve on
$M$ connecting $p_1$ and $p_2$, then it is easy to see that $\gamma$
and $f(\gamma)$ are homologous, there is a disk $D$ on $M$ such
that $$ \partial D = f(\gamma) - \gamma.$$ Finally, we define the
difference of the actions between $p_1$ and $p_2$ to be
$$\A(p_2) - \A(p_1) = \int_D \omega.$$
We remark that if $f$ is exact symplectic or Hamiltonian, then the
above definition is independent of the choice of $\gamma$. However, if
$f$ isotopic to identity, but not Hamiltonian, then it does depend on
the homotopy class of the curve $\gamma$.

\medskip

The action difference can be extended to periodic points easily. If
$\gamma_1$ and $\gamma_2$ are periodic orbits with common period
$k$. Assume that $\gamma_1$ and $\gamma_2$ are homologous. Let
$\gamma$ be a curve from one point in the orbit of $\gamma_1$ to a
point in the orbit of $\gamma_2$ and let $D$ be a disk $D$ on $M$ such
that $$ \partial D = f^k(\gamma) - \gamma,$$ then
$$\A(\gamma_2) - \A(\gamma_1) = \frac{1}{k}\int_D \omega.$$

Again, this can be extended to two invariant measures $\mu_1$ and
$\mu_2$. For this, we need to define the homology class of invariant
measures. This turns out to be exactly the rotation vectors, to be
defined next. Two invariant measures are said to be homologous if they
have the same rotation vectors.

\subsection{Rotation vectors}

Let $f: M \ra M$ be a symplectic diffeomorphism, isotopic to
identity. Let $\gamma = \{p_0, p_1, \ldots, p_k=p_0\}$, be a periodic
orbit, where $f^i(p_0) = p_i$, for $i = 1, 2,\ldots, k$. Let
$\tilde{\gamma}(t)$ be a closed curve on $M$ obtained by isotopy, with
$\tilde{\gamma}(i) = p_i$, for $i = 1, 2,\ldots, k$. We define the
rotation vector of $\gamma$ to be the homology class of
$\tilde{\gamma}$ divided by its period,
$$\rho(\gamma) = \frac{1}{k} [\tilde{\gamma} ] \in H_1(M, \R).$$

To generalize the concept of rotation vector, for any closed 1-form
$\alpha$ on $M$, we define the bi-linear form $<\cdot, \, \cdot >^*$ by 
$$<\gamma, \, \alpha>^* = \frac{1}{k}\oint_{\tilde{\gamma}} \alpha.$$
It is easy to see that above pairing depends only on the homology
class of $\tilde{\gamma}$, in $H_1(M, \R)$, and the cohomology class
of $\alpha$, in $H^1(M, \R)$. This paring equivalently defines, for
any periodic orbit $\gamma$, the rotation vector
$\rho(\gamma) \in H_1(M, \R)$, by the following equation:
$$< \rho(\gamma), [\alpha]> = <\gamma, \alpha>^* =
\frac{1}{k}\oint_{\tilde{\gamma}} \alpha$$ where the first pairing is
the canonical pairing between homology and cohomology of the manifold
$M$.

This definition of rotation vector can be naturally extended to $f$-invariant
measures in $\M_f$. Fix a closed 1-form $\alpha$, for any point $x \in
M$, let $\tilde{\gamma}(t, x), t\in \R$ be the curve in $M$ by connecting
orbit of $x$ by the isotopy in such a way that $\tilde{\gamma}(i, x) =
f^i(x)$. Define, if exists,
$$\rho(x, \alpha) = \lim_{T \ra \8} \frac{1}{T} \int_{\tilde{\gamma}(t, x): t \in [0,
  T]} \alpha.$$
It is easy to see that $\rho(x, \alpha) = \rho(x, \alpha')$, if
$[\alpha] = [\alpha'] \in H^1(M, \R)$.  This is because that if
$\alpha$ is exact, $\alpha = dS$ for some 
function $S: M \rightarrow \R$, then 
$$\rho(x, dS) =\lim_{T \ra \8} \frac{1}{T} \int_{\gamma(t, x): t \in [0,
  T]} dS = \lim_{T \ra \8} \frac{1}{T} (S(\gamma(T)) - S(x)) =0.$$

For any invariant measure $\mu \in \M_f$, by Birkhoff Ergodic Theorem,
for any fixed $\alpha$, for $\mu-a.e. \, x \in M$, the limit exists,
$\rho(x, \alpha)$ is 
well-defined. Therefore, it is also well-defined for a finite set of
basis vectors $[\alpha]$ in $H^1(M, \R)$. Hence, for $\mu -a.e. \, x \in M$,
$\rho(x, \alpha)$ is well-defined for {\em all}\/ closed 1-forms
$\alpha$. Moreover, by Birkhoff Ergodic Theorem,
$$\int \rho(x, \alpha) d\mu = \int \left(\int_{\gamma(t, x): t \in [0, 1]}
  \alpha \right)
d \mu$$

The above equation is linear in both $\alpha$ and $\mu$, it depends
only on the cohomology class of $\alpha$, therefore, it define a
pairing between $\mu$ and cohomology elements in $H^1(M, \R)$. This
pairing defines the rotation vector  $\rho(\mu) \in H_1(M, \R)$.

\begin{dfn}
For any $\mu \in M_f$ and any closed 1-form $\alpha$, $[\alpha] \in
H^1(M, \R)$, the {\em rotation vector} of $\mu$, $\rho(\mu) \in H_1(M, \R)$ is defined
by the following equation  
$$< \rho(\mu), [\alpha]>  = \int \left(\int_{\gamma(t, x): t \in [0, 1]}
  \alpha \right)
d \mu \in \R$$
where the left hand side  is
the canonical pairing between homology and cohomology of the manifold
$M$.
\end{dfn}

\medskip

As an example, if $f$ is a Hamiltonian diffeomorphism, then
$\omega^n \in \M_f$ and if $M$ is compact without boundary,
then
$$\rho(\omega^n) = 0.$$

This is not true in general for non-Hamiltonian symplectic
diffeomorphisms. An easy
example is $T_{(a, b)}: \T^2 \rightarrow \T^2$ given by
$$T_{(a, b)} (x, \, y) = (x + a, \, y + b) \mbox{ mod } \Z^2 ,$$ for
$(a, b) \notin \Z^2$.

This is also not true in general for Hamiltonian diffeomorphisms on
manifold with boundaries, for example the annulus, $\AN$.

\medskip

We now return to actions on invariant measures. As we have hinted
before, the action we defined depends on the choice of 1-form
$\beta$. It turns out that this dependence, interestingly, is closely
related to the rotation vector we just defined. Let $\alpha$ be a
closed 1-form on $M$, let
$$\tilde{\beta} = \beta + \alpha,$$
then $d\tilde{\beta} = d\beta = \omega$. Let $\tilde{g}$ and $g$ be
action functions defined by $\tilde{\beta}$ and $\beta$
respectively. To fix the integration constants, pick a point $x_0 \in
M$, let
$${g}(x)  = \int_{x_0}^x ( f^*({\beta}) - \beta)$$
and likewise $$\tilde{g}(x)  = \int_{x_0}^x ( f^*(\tilde{\beta}) -
\tilde{\beta)} = g(x) + \int_{x_0}^x (f^* \alpha - \alpha)$$

All the integrals are path independent. For any invariant measure
$\mu$, its action under $\tilde{g}$,
\be \widetilde{\A}(\mu) &=& \int_M \tilde{g} d\mu = \int_M g d\mu + \int_M \left(
\int_{x_0}^x (f^* \alpha - \alpha) \right) d\mu \nn \\
&=& \A(\mu) + \int_M \left( \int_{\gamma(t, x): t \in [0, 1]} \alpha
\right) d\mu - \int_M \left( \int_{\gamma(t, x_0): t \in [0, 1]} \alpha
\right) d\mu \nn \\
&=& \A(\mu) + <\rho(\mu), \alpha> - C_{x_0, \alpha},
\nn \ee
where $C_{x_0, \alpha}$ is a constant depending on $x_0$ and $\alpha$, 
but not on $\mu$. Here we used the Stokes' theorem and Birkhoff
Ergodic Theorem.

The constant $C_{x_0, \alpha}$ is zero, if our refence point $x_0$ is
a contractible fixed point.

Since the action is typically used in the relative sense, for comparison
between two different invariant measures, we have the
following proposition.

\begin{prop}
  Let $\mu_1, \mu_2 \in \M_f$ be two invariant probability measures and let
  $\widetilde{\A}$ and $\A$  be two actions defined by 1-forms
  $\tilde{\beta}$ and $\beta$ respectively. Let $\alpha =
  \tilde{\beta} - \beta$, then
  $$\widetilde{\A}(\mu_1) - \widetilde{\A}(\mu_2) = \A(\mu_1) -
  \A(\mu_2) + <\rho(\mu_1) - \rho(\mu_2), \, [\alpha]>$$

In particular, if two invariant measures $\mu_1$ and $\mu_2$ have exactly the
  same rotation vector, then the action difference $\A(\mu_1) -
  \A(\mu_2)$ is independent of any choice of 1-form $\beta$ and the
  integration constant.

  \label{prop4}\end{prop}

  \subsection{Actions on the annulus}

  We now restrict ourselves to two specific spaces, disk $\D$ and
  annulus $\AN$. First let $f$ be an orientation-preserving,
  area-preserving diffeomorphism of a unit disk $\D$ in
  $\R^2$. Clearly, $f$ is exact symplectic. The first homology of $\D$
  is trivial, so there is no ambiguity in the the action function. The
  action, particularly the difference in actions of two invariant
  measures, has a very clear geometric meaning. Take for example the simple
  case of two fixed points. Let $p_1$ and $p_2$ be two points in $\D$
  fixed by $f$. Pick any curve $\gamma$ connecting $p_1$ to
  $p_2$. There is a unique signed disk $U \subset \D$ such that $$\partial U
  = f(\gamma) - \gamma.$$
  Then simply, we have
  $$\A(p_2) - \A(p_1) = \int_U \omega_\D,$$
  where $\omega_\D$ is the nomalized standard area form on $\D$.

  If $p_1$ and $p_2$ are periodic points of period $k$, then we
  consider $f^k$, then the difference in action with respect to $f$ is that of
  $f^k$ divided by $k$. As for two invariant ergodic measures $\mu_1$
  and $\mu_2$, we can take generic points for these measures,
  approximate them by periodic points, then take limits.

  \
  
  The situation for the annulus is a little more complicated. Let $f$ be
  an area-preserving, orientation preserving diffeomoprhism on
  $\AN=S^1\times[0,1]$, isotopic to identity. It is easy to see that
  $f$ is exact symplectic. Suppose $p_1, p_2 \in \AN$ are two fixed
  points of $f$, and let $\gamma$ be a curve on $\AN$ connecting $p_1$
  and $p_2$. If $p_1$ and $p_2$ have exactly the same rotation number,
  then their orbits are homologous, this will be the same case as for
  the disks, their action difference in action
  is $$\A(p_2)-\A(p_1)=\int_D\omega,$$ where $D$ is a disk on $\AN$
  such that $ \partial D = f(\gamma) - \gamma$. This action difference
  is independent of the choice of $\beta$ in the definition of action
  function.

  If $p_1$ and $p_2$ are not homologous (i.e. when they have
  different rotation numbers), now $f(\gamma)$ and $\gamma$ does not
  bound an area, there is no natural way to define the region $D$ as
  in the previous case. By Proposition \ref{prop4}, the choice of
  1-form $\beta$ makes difference in the action. To remove the
  ambiguities,  we can collapse the boundary $A_0$ to a point,
  i.e. consider the quotient space $\AN/\sim$, where the equivalence
  class is defined by $(x_1,0)\sim (x_2,0)$ for all $x_i\in S^1$. The
  resulting space is the closed unit disk $\D$, and we still use $f$
  to denote the corresponding map on $\D$ and $\omega_\D$ the
  corresponding area form. For any fixed points $p_1, p_2$ of
  $f:\D \to \D$, the curves $f(\gamma)$ and $\gamma$ always bound a region $U$,
  i.e. $\partial U = f(\gamma) - \gamma$, thus the above defined
  action difference on $\AN$ is the same as $\int_U\omega_\D$,
  independent of  whether $p_1$ and $p_2$ have the same rotation number.

  Equivalently, if we add the boundary component
  $A_0=S^1 \times \{0\}$ to $\gamma$ and $f(\gamma)$ in $\AN$, then together
  they always bound a region, i.e. there is a signed region $U$ such
  that $\partial U = f(\gamma) - \gamma+ k A_0$, for some integer $k$
  depending on the rotation numbers of $p_1$ and $p_2$.  We can
  define $$\A(p_2)-\A(p_1)=\int_U\omega.$$

  The above geometric construction can be achieved by making proper
  choice of the primitive of $\omega$ on the annulus. Let $(x, y)$ be
  the natural coordinate system on $\AN = S^1 \times [0, 1]$, the
  standard area form is $\omega=dy\wedge dx$. We choose and fix $\beta
  = y dx$. This 1-form is exactly the pullback of the 1-form on the
  disk $\D$ with respect to collapsing of the bounday component $A_0$. 

  By fixing the 1-form, $\beta = y dx$, we have fixed the action
  function $g$ up to a constant.  For any diffeomorphism $f:\AN\to\AN$
  and invariant measures $\mu_1$ and $\mu_2$ in $\M_f$, the action
  difference of any two invariant measures
  $$\A(\mu_2)-\A(\mu_1)=\int_\AN gd\mu_2-\int_\AN gd\mu_1$$
  is well-defined.

\section{Proof of the main theorem}

We first prove some special cases of Theorem \ref{thm:main}. Notice that $f$ restricted to each of the boundary $\AN_0=S^1 \times \{0\}$ and $\AN_1=S^1 \times \{1\}$ is an orientation preserving circle diffeomorphism. Let $\rho_0$ and $\rho_1$ be the rotation numbers on the boundaries respectively. If $\rho_i=\frac{p}{q}$ $(i=0,1)$ is rational, then there is a periodic point with period $q$, in particular, it supports an atomic invariant measure; otherwise, there is an invariant measure supported on the boundary.

\begin{lem}
\label{lem:bdry}
Let $f$ be an area-preserving, orientation preserving diffeomoprhism on $\AN$, isotopic to identity. Let $\mu_0, \mu_1\in\M_f$ be invariant measures with supports in $\AN_0$ and $\AN_1$ respectively, and $|\A(\mu_0) -\A(\mu_1)| \neq 0$. Then there exists an interval with length $|\A(\mu_0) -\A(\mu_1)|$ such that for any rational number $\frac{p}{q}$ in the interval, there are at least two distinct periodic orbits of period $q$ with
rotation number $\frac{p}{q}$.\end{lem}
\begin{proof}
Let $\tilde{\AN}=\R\times [0,1]$ be the standard universal cover of $\AN$ and $\tilde{f}:\tilde{\AN}\to\tilde{\AN}$ a lift of $f$. Let $p_0, p_1$ be any two points on the boundaries $\tilde{\AN}_0, \tilde{\AN}_1$, respectively, and let $\gamma$ be a simple curve connecting $p_0$ and $p_1$. Let $U$ be the region bounded by $\gamma$, $\tilde{f}(\gamma)$, $h_{p_0}$ and $h_{p_1}$, where $h_{p_i}$ is the segment from $p_i$ to $\tilde{f}(p_i)$ on the boundaries $\tilde{\AN}_i$, $(i=0,1)$ respectively. Therefore we have \[\partial U=\gamma-\tilde{f}(\gamma)+h_{p_1}-h_{p_0}.\] 
Using the Stoke's theorem and the definition of the action function we have 
\[\int_U\omega=g(p_0)-g(p_1)+\int_{h_{p_1}}\beta-\int_{h_{p_0}}\beta.\]
Since $f$ is area-preserving, $\int_U\omega$ is independent of the choice of $p_0, p_1$ and the curve $\gamma$, in particular, $\int_U\omega=\rho(\omega)$, the mean rotation number of $\tilde{f}$ for the invariant measure induced by $\omega=dy\wedge dx$. Since $\beta=y dx$, we have \[\rho(\omega)=\int_U\omega=\A(\mu_0)-\A(\mu_1)+\rho_1.\]
Therefore, the rotation set of $\tilde{f}:\tilde{\AN}\to\tilde{\AN}$ contains a closed interval with length $|\A(\mu_0)-\A(\mu_1)|$. By Franks' theorem \cite{Franks1992}, for any rational number $\frac{p}{q}$ in the interval, there are at least two distinct periodic orbits of period $q$ with
rotation number $\frac{p}{q}$. In particular, for any positive integer $q> \frac{1}{|\A(\mu_0) -\A(\mu_1)|}$, there is a rational number with denominator $q$ contained in the interval, thus there must be at least two distinct periodic orbits with period $q$. Moreover, if $q$ is prime, then it is the least period.
\end{proof}

Now, let's consider the case with one of the invariant measure supported on the boundary and the other is atomic at a fixed point.

\begin{lem}
\label{lem:bdrypt}
Let $f$ be an area-preserving, orientation preserving diffeomoprhism on $\AN$, isotopic to identity. Let $\mu_0\in\M_f$ be an invariant measure with support in $\AN_0$, and $p$ is a fixed point of $f$ such that $|\A(\mu_0) -\A(p)| \neq 0$. Then for any positive integer $q$ such that
$$q> \frac{1}{|\A(\mu_0) -\A(p)|} ,$$
$f$ has at least two distinct
periodic orbits with period $q$, and $q$ is the least period if it is a prime number.
\end{lem}
\begin{proof}
If $p$ is on $\AN_1$, then it reduces to Lemma \ref{lem:bdry}, we assume $p\notin\AN_1$. 

Let $p_0$ be a point in the support of $\mu_0$, and let $\gamma$ be a simple curve connecting $p_0$ and $p$, which does not touch $\AN_1$. Let $U$ be the region bounded by $\gamma$, $f(\gamma)$, $h_{p_0}$, where $h_{p_0}$ is the segment from $p_0$ to $f(p_0)$ on the boundary $\AN_0$. Using the Stoke's theorem and the definition of the action function we have 
\[\int_U\omega=g(p_0)-g(p)-\int_{h_{p_0}}\beta=g(p_0)-g(p).\]
Since $f$ is area-preserving, we have $$\int_U\omega=\A(\mu_0)-\A(p).$$
Collapse the boundary $\AN_1$ to a point and denote it by the point $A_1$ in the resulting disk. Since on this disk, $$A_1\notin \bigcup _{0\leq t\leq 1}h_t(\gamma),$$ where $h_t, t\in[0,1]$ denotes the isotopy of $f$, we can blow up $p$ relative to $p_1$. Remove the point $p$ and add a boundary circle $C$. Denote the resulting annulus by $\AN'$ and the corresponding map by $f':\AN'\to\AN'$. Note that $A_1$ is a contractible fixed point for $f'$, hence its rotation number is zero.

In this process the transformations are area-preserving. If $\gamma'$ is a curve connecting the boundaries of $\AN'$, the area between $\gamma'$, $f'(\gamma')$ and the two boundaries of $\AN'$ will be equal to $\int_U\omega=\A(\mu_0)-\A(p)$ of the original map $f$. Thus the mean rotation number of $f'$ is $\A(\mu_0)-\A(p)$. Since the rotation number of the fixed point $A_1$ is zero, we get that the rotation set of $f'$ contains an interval of length $|\A(\mu_0)-\A(\mu_1)|$. Therefore the conclusion follows similar to Lemma \ref{lem:bdry}.

\end{proof}

\begin{rmk}
The above result also holds if we change the assumption of $\mu_0$ to $\mu_1\in\M_f$ whose support is in $\AN_1$. In this case we have \[\int_U\omega=\A(p)-\A(\mu_1)+\rho(\mu_1).\] Here $U$ is the region such that $\partial U=\gamma-f(\gamma)+h_{p_1}$, where $\gamma$ is a simple curve from $p$ to a point $p_1$ in the support of $\mu_1$. In this case, we collapse the boundary $\AN_0$ and blow up $p$. In the resulting new annulus $\AN'$, its mean rotation number is $\int_U\omega=\A(p)-\A(\mu_1)+\rho(\mu_1)$, and the boundary $\AN_1'$ has the same rotation number $\rho(\mu_1)$ as $\AN_1$ in the old annulus. Thus the rotation set contains a closed interval with length $|\A(\mu_0)-\A(\mu_1)|$.
\end{rmk}

Next, let's consider the case where both measures are supported on periodic orbits.

\begin{lem}
\label{lem:per}
Let $f$ be an area-preserving, orientation preserving diffeomoprhism on $\AN$, isotopic to identity. Suppose that $\gamma_1$ and $\gamma_2$ are periodic orbits with common period
$k$. If $|\A(\gamma_1) - \A(\gamma_2)|\neq 0$, then for any positive integer $q$ such that
$$q> \frac{1}{|\A(\gamma_1) - \A(\gamma_2)|} ,$$
$f$ has at least two distinct
periodic orbits with period $q$, and $q$ is the least period if it is a prime number.\end{lem}

\begin{proof}
Consider the map $f^k$, thus $f^k$ has two fixed points $p_1, p_2$. We have $$\A_k(p_1)-\A_k(p_2)=k(\A(\gamma_1) - \A(\gamma_2)),$$ where $\A_k$ represents the action of $f^k$. Let $F=f^k$, it suffices to show that the rotation set of $F:\AN\to\AN$ contains a closed interval with length  $k|\A(\gamma_1) -\A(\gamma_2)|$. 

We assume $p_1, p_2$ are not on the boundaries, otherwise it will reduce to the above lemmas. Collapse $\AN_0$ to a point so that we get a closed disk, next blow up $p_2$ and add a circle $C$ to get an annulus $\AN'$. Use the standard coordinates $(x,y)\in S^1\times [0,1]$ for $\AN'$ such that the circle $C=S^1\times\{0\}$. Also let $\beta=ydx$ be the primitive for $\AN'$, then the action difference between $p_1$ and the boundary $C$ for $F':\AN'\to\AN'$ is equal to the original action difference of $p_1, p_2$ for $F: \AN\to\AN$.  Apply the proof of Lemma \ref{lem:bdrypt} to $F':\AN'\to\AN'$ completes the proof of Lemma \ref{lem:per}.

\end{proof}

Finally, let's prove the general case where $\mu_1, \mu_2$ are any two invariant measures in $\M_f$ with non-zero action difference. 

\begin{proof}[Proof of Theorem \ref{thm:main}]
By Ergodic decomposition theorem (cf. \cite{Mane1987}), it suffices to prove the theorem by assuming $\mu_1$, $\mu_2$ are ergodic invariant measures. We first assume both of the measures are non-atomic. For each fixed positive integer $q$ such that $q>\frac{1}{|\A(\mu_1) -\A(\mu_2)|}$, choose $\epsilon=\epsilon(q)>0$ so that $$q>\frac{1}{|\A(\mu_1) -\A(\mu_2)|-\epsilon}.$$

Let $p_i\in \AN$ be regular points of $\mu_i$, $i=1,2$ respectively,  then \[\lim_{n\to\infty}\frac{1}{n}\sum_{j=0}^{n-1}g(f^j(p_i))=\int_\AN g d\mu_i=\A(\mu_i),\quad i=1,2.\]
Notice that $f$ is area-preserving, thus the regular points $p_i$ are recurrent. Since both measures are non-atomic, there are small neighborhoods $U_i$ of $p_i$ and sufficiently large $k\in \N$ such that $f^k(p_i)\in U_i$ and $f^j(U_i)\cap U_i=\emptyset$ for all $j=1, \cdots, k-1$. 
We will assume $k$ is large enough so that  $$k>\frac{1}{|\A(\mu_1) -\A(\mu_2)|} +10q$$ and \[|\frac{1}{n}\sum_{j=0}^{n-1}g(f^j(p_i))-\A(\mu_i)|<\frac{\ve}{5},\quad \forall n\geq k.\]

Let $V_i\subset U_i$ be a disk containing $p_i$ and $f^k(p_i)$ and we choose an isotopy $h_t: \AN\to \AN$ such that 
\begin{itemize}
\item[(1)] $h_0=\text{id}: \AN\to \AN$,
\item[(2)] $h_t(z)=z$ for all $z$ in $\AN\setminus (V_1\cup V_2)$,
\item[(3)] $h_1(f^k(p_i))=p_i$, $i=1, 2$.
\end{itemize}
Let $f_1=h_1\circ f$ and notice that $f_1^k(p_i)=p_i$, and $f_1^j(p_i)=f^j(p_i)$ for $1\leq j\leq k-1$. Let $g_1$ be the action function for $f_1$ and $\A_1$ the corresponding action acts on invariant measures. If we choose $U_i$ sufficiently small we can make \[|\A_1(\mu_{p_i})-\A(\mu_i)|<\frac{\ve}{2},\quad i=1, 2,\] where $\mu_{p_i}$ is the measure supported on $\{p_i, f(p_i), \cdots, f^{k-1}(p_i)\}$ and each point has measure $\frac{1}{k}$. Notice that each $\mu_{p_i}$ is an invariant measure for $f_1$ since $p_i$ are $k$-periodic for $f_1$. 
Now, consider the map $f_1:\AN\to\AN$, which is isotopic to the identity and preserves orientation. Moreover, $f_1$ has two periodic points $p_1, p_2$ with period $k$ and $$|\A_1(\mu_{p_1})-\A_1(\mu_{p_2})|\geq |\A(\mu_1)-\A(\mu_2)|-\ve.$$ Then by Lemma \ref{lem:per}, $f_1$ has at least two periodic points with period $\tilde{q}$ whenever $$\tilde{q}>\frac{1}{|\A(\mu_1) -\A(\mu_2)|-\ve}.$$  Notice that $k>\frac{1}{|\A(\mu_1) -\A(\mu_2)|} +10q$. By the construction of the perturbation $h_1$ whose support is in $U_1\cup U_2$, we know that any new periodic points that may be created by the perturbation must have periods no less than $k$. Thus any periodic points of $f_1$ with much smaller periods than $k$ must be periodic points of the original map $f$ in the beginning, hence $f$ has at least two periodic points with period $q$ and proving the theorem.

If one of the measure is non-atomic and the other is atomic, we only need to do perturbation for the non-atomic measure and then the proof is similar. If both measures are atomic, this reduces to the case of Lemma \ref{lem:per}.
\end{proof}

\section{Application and examples}

In this section, as application of our main theorem, we give some
examples.

\begin{example}[Local perturbation and periodic points]
\label{ex:ptb}
Let $f: \AN\to \AN$ be a rigid irrational rotation, i.e. \[f(x, y)=(x+a, y),\quad a\in\R\setminus\Q,\] then $f$ has no periodic points. In this case, $f^*\beta-\beta=0$, thus the action function is a constant. Without loss of generality, we assume it's the zero function, then the action of any $\mu\in \M_f$ is zero. 

Consider a local perturbation, i.e. an area-preserving, orientation preserving diffeomoprhism $h:\AN\to\AN$ isotopic to identity, such that $h$ supports on a small open set $U\subset\AN$. Choose $h$ properly (cf. Example \ref{ex:disk}), we can make the mean action (i.e. the Calabi invariant) of $h$ strictly positive, thus by Proposition \ref{prop:ac_comp}, the mean action of the perturbed map $f':=f\circ h$ satisfies \[\A(f')=\A(f)+\A(h)=\A(h)\neq0.\] 
Notice that outside $U$, $h$ is the identity map, thus the action of the boundary of $f'$ is still zero. By Theorem \ref{thm:main}, for any positive integer $q$ such that
$q> \frac{1}{|\A(h)|}$, $f'$ has at least two distinct periodic orbits with period $q$.

In general, if the local perturbation changes the action, one can create periodic points, which could be very useful in many situations.
\end{example}

\begin{example}[Mean action of a monotone twist map on the disk]
\label{ex:disk}
Let $\D$ be the unit disk with standard area form $\omega_\D=\frac{1}{\pi}rdrd\theta$ and let $\beta_\D=\frac{1}{2\pi}r^2d\theta$. Let $\phi: [0,1]\to\R$ be a smooth function such that $$\phi'(r)\leq0, \quad \phi(0)>0,\quad \phi(1)=0,$$ and $\phi(r)=0$ for all $r$ near $1$. Consider the map \[h:\D\to\D\] \[(r, \theta)\mapsto(r,\theta+\phi(r)).\]
Thus \[h^*\beta_\D-\beta_\D=\frac{1}{2\pi}r^2\phi'(r)dr=dg,\] 
let \[g(r, \theta)=\frac{1}{2\pi}\int_1^rs^2\phi'(s)ds.\]
The mean action of $h$ is equal to \[\A(h)=\int g\omega_\D=\frac{1}{\pi}\int_0^1r\int_1^rs^2\phi'(s)dsdr>0.\]
We can use such a map on $U=\D_\epsilon$, the disk of small radius $\epsilon>0$, to construct the desired local perturbation in Example \ref{ex:ptb}.
\end{example}

\section*{Acknowledgement}
The first author is supported by Sun Yat-sen University start-up grant No. 74120-18841290.

\bibliographystyle{acm}

\bibliography{action}

\end{document}